\documentclass[12 pt]{article}%
\usepackage{amsmath, amsfonts, amsthm, color,latexsym}

\providecommand{\U}[1]{\protect\rule{.1in}{.1in}}
\allowdisplaybreaks[4]
\newtheorem{theorem}{Theorem}[section]
\newtheorem{proposition}[theorem]{Proposition}
\newtheorem{corollary}[theorem]{Corollary}
\newtheorem{example}[theorem]{Example}

\newtheorem{lemma}[theorem]{Lemma}
\newtheorem{final remark}[theorem]{Final Remark}

\allowdisplaybreaks[4]

\newcommand {\R}{\mathbb{R}}
\newcommand {\K} {\mathbb{K}}

\newcommand {\N} {\mathbb{N}}
\newcommand{\norma}[1]{\| #1 \|}
\newcommand{\conj}[2]{\left \{ {#1} \, : \, {#2} \right \}}

\begin{document}

\title{Lattice isomorphic Banach lattices of polynomials}
\author{Christopher Boyd and Vinícius C. C. Miranda\thanks{Supported by FAPESP (Grants 2023/12916-1 and 2025/08630-0) and FAPEMIG (Grant APQ-02622-26).\newline 2020 Mathematics Subject Classification: 46B42, 46B28, 46G25.\newline Keywords: Banach lattices, regular homogeneous polynomials, compact polynomials, weakly compact polynomials, orthogonally additive polynomials, regular nuclear polynomials}.}
\date{}
\maketitle

\begin{abstract}
  We study Díaz-Dineen's  problem for regular homogeneous vector-valued polynomials. In particular, we prove that if $E^*$ and $F^*$ are lattice isomorphic
  with at least one having order continuous norm, then $\mathcal{P}^r(^n E;
  G^*)$ and $\mathcal{P}^r(^n F; G^*)$ are lattice isomorphic for every $n\in
  \N$ and every Banach lattice $G$. We also study
  the analogous problem for the classes of regular
  compact, regular weakly compact, orthogonally additive
    and regular nuclear polynomials.
\end{abstract}

\section{Introduction and background}

In \cite{diazdineen}, J. C. Díaz and S. Dineen posed the following question: \textit{If $X$ and $Y$ are dual isomorphic Banach spaces, that is $X^\ast$ and $Y^\ast$ are isomorphic as Banach
    spaces,
    are $\mathcal{P}(^nX)$ and $\mathcal{P}(^n Y)$ isomorphic as Banach
    spaces
    ?} Although it is not known if the above question has a positive answer, some partial answers were presented in \cite{cabello, diazdineen, lassalezalduendo}.
In particular, J. C. D\'iaz and S. Dineen proved that the
problem has a positive solution whenever $X^\ast$ has both the Schur property and the approximation property (see \cite[Proposition 4]{diazdineen}). In \cite{cabello},
F. Cabello Sánchez, J. Castillo and R. García generalized Díaz-Dineen results
by providing a positive solution for this problem whenever $X$ is Arens-regular
or whenever $X$ and $Y$ are stable.
The same result was proven, independently, by S. Lassalle and I. Zalduendo in \cite{lassalezalduendo} who also addressed the same question concerning
other classes  of polynomials.
  Further, in \cite{carandolassale}, D. Carando and S.
  Lassalle extended these results to spaces of vector-valued homogeneous
polynomials.

Motivated by the increasing interest in studying homogeneous polynomials on Banach lattices (see, e.g., \cite{botelholuiz, botlumir, botmirru,  boydryansnig2, boydryansnig, boydryansnigtams, buquaterly, galmir1, libu}), the following question arises naturally: \textit{Given two dual isomorphic Banach lattices $E$ and $F$, that is $E^\ast$ and $F^\ast$ are lattice isomorphic, under which
  conditions are $\mathcal{P}^r(^n E)$ and $\mathcal{P}^r(^n F)$ lattice
  isomorphic}
for every $n \in \N$? Although our construction is inspired by
\cite{carandolassale,lassalezalduendo}, our conclusions are genuinely
order-theoretic. While the previous results provide linear isomorphisms between
spaces of homogeneous polynomials, we prove that the induced maps and their
inverses preserve positivity, yielding lattice isomorphisms between spaces of
regular polynomials. Here, regularity means being a difference of positive
polynomials and should not be confused with the Arens regularity notion used
in \cite{carandolassale}. In particular, we prove that if
$E^\ast$ and $F^\ast$ are lattice isomorphic with at least one (and hence the
  other) having order continuous norm, then
$\mathcal{P}^r(^n E; G^*)$ and $\mathcal{P}^r(^n F; G^*)$ are lattice isomorphic for every $n \in \N$ and every Banach lattice $G$. 
  We observe that, in our setting, the assumption of Arens regularity is replaced with the requirement that the norm on the dual is order continuous. This replacement is justified by the fact that a positive version of Arens regularity for $E$  is equivalent to the order continuity of the norm on $E^*$. Furthermore, as we will demonstrate in Proposition \ref{proparensreg}, positive Arens regularity coincides with classical Arens regularity and with
  positive symmetric Arens regularity, highlighting a noteworthy contrast with the Banach space case.

As an application, we
obtain conditions so that
$\mathcal{P}_{\cal K}^r(^n E; G^*)$ (resp. $\mathcal{P}_{\cal W}^r(^n E; G^*)$) and $\mathcal{P}_{\cal K}^r(^n F; G^*)$ (resp. $\mathcal{P}_{\cal W}^r(^n F; G^*)$) are isomorphic, where $\mathcal{P}_{\cal K}^r$ and $\mathcal{P}_{\cal W}^r$ denote, respectively, the classes of regular compact and regular weakly compact polynomials. We also consider the class of orthogonally additive polynomials.
In Section 3, we investigate this problem for a new
class of polynomials. Specifically, inspired by ideas from \cite{blancolal}, we define the class of regular nuclear $n$-homogeneous polynomials, denoted by $\mathcal{P}_{\cal N}^r$, and prove that if $G$ is Dedekind complete and has
the lattice approximation property (LAP), and $E^*$ and $F^*$ are lattice
isomorphic and both have the LAP, then  $\mathcal{P}_{\cal N}^r(^n E; G)$ and
$\mathcal{P}_{\cal N}^r(^n F; G)$ are lattice isomorphic for every $n \in \N$.

Before proceeding, we introduce some definitions and notation.
Throughout this paper, $X$ and $Y$ denote Banach spaces, $E$ and $F$ denote Banach lattices, $E^+$, $B_E$ and $S_E$ denote, respectively, the positive cone, the closed unit ball, and the unit sphere of $E$. Given Banach spaces $X_1, \dots, X_n$ and $Y$, the Banach space of all continuous $n$-linear operators $A\colon X_1 \times \cdots \times X_n \to Y$ is denoted by $\mathcal{L}(X_1, \dots, X_n; Y)$. In the case that $Y = \K$, we simply write $\mathcal{L}(X_1, \dots, X_n)$. Given Banach lattices $E_1, \dots, E_n$ and $F$, an $n$-linear operator $A\colon E_1 \times \cdots \times E_n \to F$ is said to be positive if $A(x_1, \dots, x_n) \geq 0$ for all $x_1 \in E_1^+, \dots, x_n \in E_n^+$. The difference of two positive $n$-linear operators is called a regular $n$-linear operator, and the set of all regular $n$-linear operators from $E_1 \times \cdots \times E_n $ into $F$ is denoted by $\mathcal{L}^r(E_1, \dots, E_n; F)$. Whenever $F$ is Dedekind complete, $\mathcal{L}^r(E_1, \dots, E_n; F)$ is a Banach lattice with the regular norm $\norma{A}_r = \norma{|A|}$, where $|A|$ denotes the absolute value of the regular $n$-linear operator $A\colon E_1 \times \cdots \times E_n \to F$. The \textit{positive projective tensor product}, or the \textit{Fremlin tensor product} of $E_1, \dots, E_n$, denoted by $E_1 \widehat{\otimes}_{|\pi|} \cdots \widehat{\otimes}_{|\pi|} E_n$, is the Banach lattice generated by considering the completion of the vector lattice tensor product $E_1 \overline{\otimes} \cdots \overline{\otimes} \, E_n$ with respect to the \textit{positive projective tensor norm}:
$$ \norma{u}_{|\pi|} = \inf \conj{\sum_{i=1}^k \norma{x_{i}^1}\cdots \norma{x_i^n}}{|u| \leq \sum_{i=1}^k x_i^1 \otimes \cdots \otimes x_i^n, \, x_i^j \in E_j^+}.
$$
For each $T \in \mathcal{L}^r(E_1, \dots, E_n; F)$, there exists a unique bounded linear operator $T^\otimes$ from  $E_1 \widehat{\otimes}_{|\pi|} \cdots \widehat{\otimes}_{|\pi|} E_n$ into $F$, called the linearization of $T$, such that $T^\otimes (x_1 \otimes \cdots \otimes x_n) = T(x_1, \dots, x_n)$ for all $x_1 \in E_1, \dots, x_n \in E_n$. Moreover, the map
$$ T \in \mathcal{L}^r(E_1, \dots, E_n; F) \mapsto T^\otimes \in \mathcal{L}^r(E_1 \widehat{\otimes}_{|\pi|} \cdots \widehat{\otimes}_{|\pi|} E_n; F) $$
is an isometric isomorphism and a lattice homomorphism (see \cite[Proposition 3.3]{bubuskes}).

An $n$-homogeneous polynomial $P\colon E \to F$ between Riesz spaces is positive  if its associated symmetric multilinear operator $T_P\colon E^n \to F$ is positive. The difference of two positive $n$-homogeneous polynomials is called a regular homogeneous polynomial, and the set of such
polynomials is denoted by $\mathcal{P}^r(^nE; F)$. When $F$ is the scalar field we simply write $\mathcal{P}^r(^nE)$. If $E$ and $F$ are Banach lattices with $F$ Dedekind complete, then $\mathcal{P}^r(^nE; F)$ is a Banach lattice with the regular norm $\norma{P}_r = \norma{|P|}$, where $|P|$ denotes the absolute value of the regular $n$-homogeneous polynomial $P\colon E \to F$. 
The $n$-fold positive projective symmetric tensor product of $E$, denoted by
$\widehat{\otimes}_{n, s, |\pi|} E$,
is a Banach lattice endowed with the positive projective symmetric tensor
norm: 
$$
\norma{u}_{s, |\pi|} = \inf \conj{\sum_{i=1}^k \norma{x_{i}}^n}{|u| \leq
  \sum_{i=1}^k \otimes^n x_i, \, x_i \in E^+},
$$
where  $\otimes^n x = x \, \otimes \stackrel{n}{\cdots} \otimes \, x$ for every $x \in E$.  Let  $\theta_n\colon E \to \widehat{\otimes}_{n, s, |\pi|} E$ be the canonical $n$-homogeneous polynomial given by $\theta_n(x) =\otimes^n x$. We note that $\theta_n$
is a lattice homomorphism. For every $P \in \mathcal{P}^r(^n E; F)$ there
exists a unique regular linear operator $P^\otimes\colon \widehat{\otimes}_
{n, s,|\pi|} E \to  F$, called the linearization of $P$, such that $P(x) =
P^\otimes(\theta_n( x))$ for every $x \in E$. 
 The operator
$$\Phi\colon \mathcal{L}^r\left(\widehat{\otimes}_{n, s, |\pi|} E; F\right) \to \mathcal{P}^r(^n E; F)~,~ \Phi(T) = T \circ \theta_n, \quad T \in \mathcal{L}^r(\widehat{\otimes}_{n, s, |\pi|} E; F),
 $$ is an isometric isomorphism and a lattice homomorphism \cite[Proposition 3.4]{bubuskes}.

 We refer the reader to \cite{alip, meyer} for background on Banach lattices, to \cite{fabian} for Banach space theory, to \cite{dineen} for polynomials on Banach spaces, to \cite{bubuskes, loane} for positive/regular multilinear operators or polynomials, and to \cite{bubuskes, frem1, frem2}
for tensor products between vector lattices.

\section{Díaz-Dineen problem for regular polynomials}

We begin this section by recalling that the \textit{first Arens adjoint} of $A \in \mathcal{L}(X_1, \cdots,X_n; Y)$ is the $n$-linear map $A^*\colon
Y^* \times X_1 \times \cdots \times X_{n-1} \to X_n^*$ defined by 
$$ A^*(y^*, x_1, \dots, x_{n-1})(x_n) = y^*(A(x_1, \dots, x_n) ).$$
The second Arens adjoint $A^{**}$ of $A$ is the first Arens adjoint of $A^*$, that is the $n$-linear map $A^{**}\colon X_n^{**} \times Y^* \times
  X_1 \times \cdots \times X_{n-2} \to X_{n-1}^*$
  given by
  $$ A^{**}(x_n^{**}, y^*, x_1, \dots,x_{n-2})(x_{n-1}) = x_n^{**}(A^*(y^*, x_1,
  \dots, x_{n-1})). $$
  This process is repeated until we obtain $\overline{A} : =  A^{(n+1)*}\colon
  X_1^{**}
  \times \cdots \times X_n^{**} \to Y^{**}$, the $(n+1)$-\textit{Arens adjoint} of $A$  which is called the \textit{Arens-extension} of $A$ (see \cite[Section 2]{boydryansnig2} for more details). 
  For the remainder of the paper, we shall call this extension the
  Arens extension, bearing in mind that it coincides with its Aron-Berner extension
(see \cite{aroncolegamelin}). 
 The \textit{transpose} of a bilinear map $A \colon X\times Y
   \to Z$ is the bilinear map $A^t \colon Y \times X \to Z$ given by $A^t(y,x) = A(x,y)$. We will say that $A$ is \textit{Arens-regular}
 (or simply \textit{regular}) if $A^{t***t} = A^{***}$.
It is important to note that we have two different concepts of regularity for multilinear operators, one in the lattice sense presented in the Introduction and
one in the extension sense. It follows from \cite[Theorem 2.2]{ulger}
that a bilinear form $A\colon X \times X \to \R$ is Arens-regular if and
only if the bounded linear operator $S_A\colon X \to X^*$ given by $S_A(x)(y)
= A(x,y)$ is weakly compact.

Recall that a Banach space $X$ is called Arens-regular if all linear
operators from $X$ into $X^*$ are weakly compact. To our knowledge, the first
appearance of a positive version of Arens-regularity was in \cite{buskes}, when G. Buskes and R. Page gave necessary and sufficient conditions so that every positive bilinear form $E \times E \to \R$ is Arens-regular. The following result collects
the characterizations obtained so far by G. Buskes, R. Page, G. Emmanuele, E.
Saab and P. Saab, and proves that positive
Arens-regularity does not give rise to a genuinely new concept as it is equivalent to   classical Arens-regularity.


\begin{proposition} \label{proparensreg}
    For a Banach lattice $E$, the following are equivalent: \\ 
    {\rm (1)} $E^\ast$ has order continuous norm. \\
    {\rm (2)} $E$ does not contain a sublattice isomorphic to
    $\ell_1$. \\
    {\rm (3)} Every positive bilinear map $A\colon E \times E \to \R$ is
    Arens-regular. \\
    {\rm (4)} Every positive operator $T\colon E \to E^*$ is weakly compact. \\
    {\rm (5)} Every positive symmetric bilinear map $A\colon E \times E \to
    \R$ is Arens-regular. \\
    {\rm (6)} $E$ is Arens-regular.
\end{proposition}

\begin{proof}
    The equivalences (1)$\Leftrightarrow$(2)$\Leftrightarrow$(3) follow, respectively, from \cite[Theorem 2.4.14]{meyer} and \cite[Theorem 1]{buskes}. The equivalence (3)$\Leftrightarrow$(4) follows from \cite[Theorem 2.2]{ulger} and the identification 
    $$ A\in \mathcal{L}^+(E, E; \R) \mapsto S_A \in \mathcal{L}^+(E; E^\ast). $$

The implications (3)$\Rightarrow$(5) and (6)$\Rightarrow$(3) are immediate.

(5)$\Rightarrow$(2) For the sake of contradiction, we assume that $E$ contains a sublattice  isomorphic to $\ell_1$, and so there exists a positive projection $\Pi\colon E \to \ell_1$ (see \cite[Proposition 2.3.11]{meyer}).
Let $J: \ell_1 \to E$ be the lattice embedding such that $\Pi J = id_{\ell_1}$
and
  let $B\colon \ell_1 \times \ell_1 \to \R$ be the symmetric bilinear map which is not Arens-regular presented in \cite[p. 83]{aroncolegamelin}:
    $$ B(x,y) = \sum_{j \text{ even}} \sum_{k=1}^{j-1} x_j y_k \, + \sum_{k \text{ even}} \sum_{j=1}^{k-1} x_j y_k, \quad x = (x_i)_i, y = (y_i)_i \in \ell_1.$$
  \noindent Since $B$ is positive, we may consider the positive symmetric bilinear map $A\colon E \times E \to \R$ given by $A(x,y) = B(\Pi(x), \Pi(y))$
  for all $x,y \in E$. If $S_A = \Pi^* S_B \Pi$ were weakly compact, then $J^*  S_A  J$ is also weakly compact. However, since
  \begin{align*}
      J^* S_A  J = J^* \Pi^* S_B \Pi J = (id_{\ell_1})^* \circ S_B \circ id_{\ell_1} = S_B,
  \end{align*}
    we would have that $S_B$ is weakly compact, a contradiction.

  (1)$\Rightarrow$(6) If $E^*$ has order continuous norm, then $E^*$ is a
  KB-space \cite[Theorem 2.4.14]{meyer}, and hence $E^*$ is weakly
  sequentially complete \cite[Theorem 4.60]{alip}. Thus, $E^*$ has the
  so-called property $(V^*)$ of Pe{\'l}czynski (see
  \cite[Theorem 1.7]{emmanuelle}), and by \cite[Proposition 8]{saab}, we
  obtain that $E$ is Arens-regular.
\end{proof}

It is worth mentioning that Proposition \ref{proparensreg}
shows that the
notions of Arens-regularity, positive Arens-regularity, and positive symmetric
Arens-regularity are all equivalent in the Banach lattice setting. This result
highlights an important
difference between the theory of Banach spaces and the theory of Banach
lattices. In \cite{leung}, D. Leung showed
that if $(e_i)_i$ is the canonical basis for $\ell_p$, $2\le p<\infty$,
then
  $J(e_i)'$, the dual of James space $J(e_i)$, is symmetrically Arens regular but
not Arens regular. Observe  that $J(e_i)$
is not a Banach lattice since it is a non-reflexive space containing no copies of $c_0$ or $\ell_1$ (see \cite[Theorem
    4.71]{alip}).

Our next result is an immediate consequence of Proposition \ref{proparensreg}
and \cite[p. 52]{zalduendo}.

\begin{corollary} \label{lema1}     If $E^*$ has order continuous norm, then the Arens extension $\overline{A}$ of every positive symmetric $n$-linear operator $A\colon E^n \to F$ is also symmetric.
\end{corollary}

The following construction is from \cite{carandolassale, lassalezalduendo}.
Given a bounded linear operator $\varphi\colon X^\ast \to Y^\ast$ and a symmetric
$n$-linear form $A\colon X \times \cdots \times X \to \R$, we define $\Tilde{\varphi}(A) \in \mathcal{L}(^n Y)$ by
$$ \Tilde{\varphi}(A)(y_1, \dots, y_n) = \overline{A}(\varphi^\ast(J_Y(y_1)), \dots, \varphi^\ast(J_Y(y_n))) \quad \text{for all $y_1, \dots, y_n \in Y$}. $$
In general, $\Tilde{\varphi}(A)$ is not a symmetric $n$-linear form on $Y$. However, for every $P \in \mathcal{P}(^n X)$, we can define $\overline{\varphi}(P) \in \mathcal{P}(^n Y)$ by
$$ \overline{\varphi}(P)(y) = \Tilde{\varphi}(T_P)(y, \dots, y) \quad \text{for every $y \in Y$.} $$
In the case $\varphi = J_{X^\ast}$, the morphism $\overline{\varphi}$ obtained is the Arens extension of polynomials from a Banach space $X$
to its bidual $X^{\ast \ast}$. In this case, we will use the notation $\overline{P}$ and $\overline{A}$ for $\overline{J_{X^\ast}}(P)$ and $\overline{J_{X^\ast}}(A)$, respectively. 
Now, for a symmetric $n$-linear map $A\colon X^n \to Z$, we define
$\Tilde{\varphi}(A) \in \mathcal{L}(^n Y; Z^{\ast \ast})$ by
$$ \Tilde{\varphi}(A)(y_1, \dots, y_n)(z^\ast) = \Tilde{\varphi}(z^\ast \circ A)(y_1, \dots, y_n) $$
for all $y_1, \dots, y_n \in Y$ and $z^\ast \in Z^\ast$. Moreover, for every $P \in \mathcal{P}(^n X; Z)$ we may define $\overline{\varphi}(P) \in \mathcal{P}(^n Y; Z^{\ast \ast})$ by 
$$ \overline{\varphi}(P)(y) = \Tilde{\varphi}(T_P)(y, \dots, y) \quad \text{for every $y\in Y$}. $$

In the proof of our next results, we will use the fact that the Arens extension of every positive $n$-linear map is also positive \cite[Theorem 2.2(c)]{botgar}. A proof of this result can be found in the doctoral dissertation of L.~Garcia \cite[Theorem 2.1.3]{garcia}.

\begin{theorem} \label{teorema2}
  Suppose that $E^*$ has order continuous norm. If $\varphi\colon
  E^\ast \to F^\ast$ is a lattice isomorphism, then $\overline{\varphi}\colon
  \mathcal{P}^r(^n E; G) \to \mathcal{P}^r(^n F; G^{\ast \ast})$ is a positive linear operator which is
  a linear isomorphism 
  onto its image for every $n \in \N$ and every Dedekind
  complete Banach lattice $G$.
\end{theorem}

\begin{proof}
  We begin by proving that $\overline{\varphi}(P) \geq 0$ for every positive $n$-homogeneous polynomial $P\colon E \to G$. To see this, let $P\colon E
  \to G$ be a positive $n$-homogeneous polynomial. We
    will prove that $\Tilde{\varphi}(T_P)
  (y_1, \dots, y_n) \geq 0$ in $G^{\ast \ast}$ for all $y_1, \dots, y_n \in F^+$.
   Indeed, given $y_1, \dots, y_n \in F^+$ and $0 \leq z^\ast \in G^\ast$, we have that 
    \begin{align*}
        \Tilde{\varphi}(T_P)(y_1, \dots, y_n)(z^\ast) & = \Tilde{\varphi}(z^\ast \circ T_P)(y_1, \dots, y_n) \\
        & = \overline{z^\ast \circ T_P}(\varphi^\ast(J_F(y_1)), \dots, \varphi^\ast(J_F(y_n))) \geq 0
    \end{align*}
    because  $\overline{z^\ast \circ T_P} \geq 0$ (due to $z^\ast \circ T_P \geq 0$), $\varphi^\ast \geq 0$ (due to $\varphi \geq 0$) and $J_F \geq 0$
    (see \cite[Proposition 1.4.5(ii)]{meyer}). This proves that $\Tilde{
      \varphi}(T_P)\colon F^n \to G^{\ast \ast}$ is a positive
    $n$-linear map, and consequently $\Tilde{\varphi}(T_P)^s$ is also a positive $n$-linear map (see \cite[Proposition 1.6]{mujica}). Since $\Tilde{
        \varphi}(T_P)^s$ is the symmetric $n$-linear form associated to
      $\overline{\varphi}(P)$, we conclude that $\overline{\varphi}(P)$ is a
      positive $n$-homogeneous polynomial.
    Therefore, $\overline{\varphi}\colon \mathcal{P}^r(^n E; G) \to
           \mathcal{P}^r(^n F; G^{\ast \ast})$ is positive.
           Since $\varphi^{-1}\colon F^\ast \to E^\ast$ is a positive operator
           (see \cite[Theorem 2.15]{alip}), we also have that 
           $\overline{\varphi^{-1}}\colon \mathcal{P}^r(^n F; G^{**}) \to
           \mathcal{P}^r
      (^n E; G^{****})$  is also a positive operator. 

           Finally, since $E^*$ has order continuous norm, we get by
           Corollary \ref{lema1} that $\overline{T_P}$ is symmetric for every positive $n$-homogeneous polynomial $P\colon E \to G$. Thus, \cite[Lemma 1.1(b)]{carandolassale} yields that
    $\overline{\varphi^{-1}} \circ \overline{\varphi} (P) = P$ holds for every $P \in \mathcal{P}^+(^n E; G)$, and hence 
     $$\overline{\varphi^{-1}} \circ \overline{\varphi} (P) = \overline{\varphi^{-1}} \circ \overline{\varphi} (P^+ - P^-) =  \overline{\varphi^{-1}} \circ \overline{\varphi}(P^+) - \overline{\varphi^{-1}} \circ \overline{\varphi}(P^-) = P^+ - P^- = P$$ 
    for every $P \in \mathcal{P}^r(^n E; G)$, which proves that $\varphi$ is a linear isomorphism onto its image.   
\end{proof}

The first limitation of Theorem \ref{teorema2}
is that $\overline{\varphi}(P)$ may not take its values in
$G$ (see \cite[p. 285]{carandolassale}). The second limitation is that we
cannot ensure that the image of $\overline{\varphi}$ is a vector sublattice of $\mathcal{P}^r(^n F; G^{**})$. However, when we consider polynomials taking values in a dual Banach lattice, we obtain the following.

\medskip

\begin{theorem} \label{teorema3} 
  If $E^\ast$ and $F^\ast$ are lattice isomorphic with at least one
   (and hence the other) having an order continuous norm,
  then $\mathcal{P}^r(^n E; G^\ast)$ and $\mathcal{P}^r(^n F; G^{\ast})$ are lattice isomorphic for every $n \in \N$ and for every Banach lattice $G$. 
\end{theorem}

\begin{proof} Let $\varphi\colon E^\ast \to F^\ast$ be the lattice isomorphism
  given by the assumption.
  We proceed as in \cite{carandolassale}: define
  $\overline{\varphi_{G}}\colon\mathcal{P}(^n E; G^\ast) \to \mathcal{P}(^n F; G^{\ast})$ by
    $$ \overline{\varphi_{G}}(P)(y)(z) = \overline{\varphi}(J_{G}(z) \circ P)(y) $$
    for all $P \in \mathcal{P}(^n E; G^\ast)$, $y \in F$ and $z \in G$. By the definition of $\overline{\varphi}$, we get that
      $$ \overline{\varphi_{G}}(P)(y)(z) = \overline{P}(\varphi^\ast (J_F(y))) (z) = \overline{T_P}(\varphi^\ast (J_F(y)), \dots, \varphi^\ast (J_F(y)))(z).$$
    In particular, given $P \in \mathcal{P}^+(^n E; G^\ast)$, the $n$-linear map
    $A\colon F^n \to G^\ast$ given by 
    $$ A(y_1, \dots, y_n)(z) = \overline{T_P}(\varphi^\ast (J_F(y_1)), \dots, \varphi^\ast (J_F(y_n)))(z) $$
    is positive, symmetric (because $\overline{T_P}$ is symmetric by the assumption and Corollary \ref{lema1})    
    and satisfies $\overline{\varphi_G}(P)(y) = A(y, \dots, y)$. Thus  $ \overline{\varphi_{G}}(P) \geq 0$ for every positive $n$-homogeneous polynomial
    $P\colon E \to G^*$. Therefore, we may consider the positive linear operator $\overline{\varphi_{G}}$ from  $\mathcal{P}^r(^n E; G^\ast)$ into $  \mathcal{P}^r
    (^n F; G^{\ast})$. Analogously, since $\varphi^{-1}\colon F^\ast \to
    E^\ast$ is a positive operator (see \cite[Theorem 2.15]{alip}), we get
    that $\overline{\varphi_{G}^{-1}}\colon \mathcal{P}^r(^n F; G^\ast) \to
    \mathcal{P}^r (^n E; G^\ast)$ is also a positive operator.

    Finally, since $E^*$ and $F^*$ have order continuous norms, we get by the same argument used in the proof of \cite[Theorem 1.3]{carandolassale} that $\overline{\varphi_{G}^{-1}} \circ \overline{\varphi_{G}} (P) = P$ and $\overline{\varphi_{G}} \circ \overline{\varphi_{G}^{-1}} (Q) = Q$ hold for all $P \in \mathcal{P}^+(^n E; G^\ast)$ and $Q \in \mathcal{P}^+(^n F; G^\ast)$. Consequently, $$\overline{\varphi_{G}^{-1}} \circ \overline{\varphi_{G}} (P) = \overline{\varphi_{G}^{-1}} \circ \overline{\varphi_{G}} (P^+ - P^-) =  \overline{\varphi_{G}^{-1}} \circ \overline{\varphi_{G}}(P^+) - \overline{\varphi_{G}^{-1}} \circ \overline{\varphi_{G}}(P^-) = P^+ - P^- = P$$ holds for every $P \in \mathcal{P}^r(^n E; G^\ast)$, and analogously $\overline{\varphi_{G}} \circ \overline{\varphi_{G}^{-1}} (Q) = Q$ holds for every $Q \in \mathcal{P}^r(^n F; G^\ast)$. This proves that $(\overline{\varphi}_{G})^{-1} = \overline{\varphi_{G}^{-1}}$.     By \cite[Theorem 2.15]{alip}, we conclude that $\overline{\varphi_{G}} \colon \mathcal{P}^r(^n E; G^\ast) \to \mathcal{P}^r(^n F; G^\ast)$ is a lattice isomorphism.
\end{proof}

An immediate application of Theorem \ref{teorema3} is the following:

\begin{corollary} \label{cortensor}
    If $E^*$ and $F^*$ are lattice isomorphic with at least one having order continuous norm, then $\left ( \widehat{\otimes}_{n, s, |\pi|} E\right ) ^*$ and $\left ( \widehat{\otimes}_{n, s, |\pi|} F\right ) ^*$ are lattice isomorphic for every $n \in \N$.
\end{corollary}

Let us give an example:

\begin{example} \label{exteo}
  In \cite{lacey}, E. Lacey constructed a nonatomic AM-space $X$ that is linearly isomorphic to $C \langle 1, \alpha \rangle$, where $\alpha \geq \omega$ is
  a countable ordinal and $\langle 1, \alpha \rangle$ denotes the order interval of ordinals between $1$ and $\alpha$ with the usual order topology. Moreover, there exists an isometric isomorphism $R \colon X^* \to \ell_1$ that is also order preserving, which means that $R\varphi
    \geq
    0$ in $\ell_1$ if and only if $\varphi \geq 0$ in $X^*$. This
    means that $R$ preserves the 
  lattice structure, and hence $R$ is a lattice isomorphism (see \cite[p. 5]{laceybook} or \cite[Example 3.1]{gao}), proving that $X^*$ and $\ell_1$ are lattice isomorphic. Let us check that $c_0^\ast$ and $c^\ast$ are also lattice isomorphic to $\ell_1$: \\
{\rm (i)} Let $S\colon \ell_1 \to c_0^*$ be the linear isomorphism given by $S(b)(x) = \displaystyle \sum_{j=1}^\infty b_j x_j$ for all $b = (b_j)_j \in \ell_1$ and $x = (x_j)_j \in c_0$. It is easy to see that both $S$ and $S^{-1}$ are positive operators, and so $S$ is a lattice isomorphism by
\cite[Theorem 2.15]{alip}. \\
{\rm (ii)} Let $T\colon \ell_1 \to c^*$ be the linear isomorphism given by $T(b)(x) = b_1 \displaystyle \lim_{n \to \infty} x_n + \sum_{j=2}^\infty b_j x_{j-1}$ for all $b = (b_j)_j \in \ell_1$ and $x = (x_j)_j \in c$. Clearly, $T \geq 0$. The inverse map $T^{-1}\colon c^* \to \ell_1$ is  given by 
$T^{-1}(\varphi) = (\varphi(e) - \displaystyle \sum_{j=1}^\infty \varphi(e_j), \varphi(e_1), \varphi(e_2), \dots)$, where $e = (1,1,1, \dots)$.
Indeed, for every $x = (x_j)_j \in c$, we have that $x = (\displaystyle \lim_{n \to \infty} x_n)  \, e \, + \, \sum_{j=1}^\infty (x_{j} - \lim_{n \to \infty} x_n)e_{j},  $ and hence 
\begin{align*}
    T(\varphi(e) - \sum_{j=1}^\infty \varphi(e_j), \varphi(e_1), \varphi(e_2), \dots)(x) & = (\varphi(e) - \sum_{j=1}^\infty \varphi(e_j))\displaystyle \lim_{n \to \infty} x_n  +  \sum_{j=1}^\infty \varphi(e_j)x_j \\
    & = \varphi(e) \displaystyle \lim_{n \to \infty} x_n +  \sum_{j=1}^\infty (x_j -  \lim_{n \to \infty} x_n) \varphi(e_j) \\
    & =  \varphi \left ( \lim_{n \to \infty} x_n \, e + \sum_{j=1}^\infty (x_j -  \lim_{n \to \infty} x_n) e_j  \right )\\
    & = \varphi(x).
\end{align*}
To see that $T^{-1} \geq 0$, it is enough to observe that $\varphi(e) -  \displaystyle \sum_{j=1}^\infty \varphi(e_j) \geq 0$ holds for all $\varphi \geq 0$, because the positive cone is norm-closed and $\varphi(e) - \displaystyle \sum_{j=1}^n \varphi(e_j) \geq 0$ for every $n \in \N$. Finally, since both $T$ and $T^{-1}$ are positive operators, we get from \cite[Theorem 2.15]{alip} that $T$ is a lattice isomorphism. 

Therefore, it follows that $c_0^*$, $c^*$ and $X^*$ are all lattice isomorphic, while $c_0$, $c$ and $X$ are not lattice isomorphic ($c$ is not Dedekind complete, $c_0$ has an order continuous
norm and $X$ is nonatomic, while both $c$ and $c_0$ are atomic). By Theorem \ref{teorema3}, $\mathcal{P}^r(^n X; G^\ast)$, $\mathcal{P}^r(^n c_0; G^\ast)$ and $\mathcal{P}^r(^n c; G^\ast)$ are all lattice isomorphic for every $n \in \N$
and for every Banach lattice $G$.
\end{example}

An $n$-homogeneous polynomial $P\colon E\to F$ is said to be compact (resp. weakly compact) if $P(B_E)$ is a relatively
  compact (resp. relatively weakly compact) subset of $F$.
If $\mathcal{P}_{\cal K}^+(^n E; F)$ denotes the set of all positive compact $n$-homogeneous polynomials from $E$ into $F$, we define $\mathcal{P}_{\mathcal{K}}^r(^n E; F) = \text{span } \mathcal{P}_{\mathcal{K}}^+(^n E; F)$. Analogously, we define $\mathcal{P}_{\mathcal{W}}^r(^n E; F) = \text{span } \mathcal{P}_{\mathcal{W}}^+(^n E; F)$, where 
 $\mathcal{P}_{\cal W}^+(^n E; F)$ denotes the set of all positive  weakly compact $n$-homogeneous polynomials from $E$ into $F$. It was proved in \cite{botmirru} that under certain conditions $\mathcal{P}_{\mathcal{K}}^r(^n E; F)$ and $\mathcal{P}_{\mathcal{W}}^r(^n E; F)$ are Banach lattices with the regular norm.


\begin{proposition} \label{cor1}
  Let $E, F, G$ be Banach lattices. Suppose that there exists a lattice isomorphism $\varphi\colon E^\ast \to F^\ast$, $E^*$ has order continuous norm
  and $G^\ast$ is atomic with order continuous norm or it is an AL-space. 
    Then, $\mathcal{P}_{\mathcal{K}}^r(^n E; G^\ast)$ and $\mathcal{P}_{\mathcal{K}}^r(^n F; G^\ast)$ are lattice isomorphic.
\end{proposition}

\begin{proof}
  We begin recalling from \cite[Example 2.3]{botmirru} that, under the assumption on $G^\ast$, $\mathcal{P}_{\mathcal{K}}^r(^n E; G^\ast)$ and $\mathcal{P}_{\mathcal{K}}^r(^n F; G^\ast)$ are closed sublattices of $\mathcal{P}^r(^n E;
  G^\ast)$ and $\mathcal{P}^r(^n F; G^\ast)$, respectively. Letting $\overline{\varphi_{G}}\colon  \mathcal{P}^r(^n E; G^\ast) \to \mathcal{P}^r(^n F; G^{\ast})$ be the lattice isomorphism defined in the proof of Theorem \ref{teorema3}, we claim that $\overline{\varphi_G}(P)$ is a compact polynomial for every $P \in \mathcal{P}_{\mathcal{K}}^+(^n E; G^\ast)$.

  To prove that $\overline{\varphi_G}(P)$ is compact, let $(y_k)_k \subset B_F$. In particular, $(\varphi^\ast (J_F(y_k)))_k$ is a bounded sequence in $E^{\ast \ast}$. Since $P\colon E \to G^*$
  is a compact $n$-homogeneous polynomial, it follows from \cite[Proposition 2.1]{aronberner} that $\overline{P}\colon E^{**} \to G^*$ is also a compact
  polynomial, and hence
    $(\overline{P}(\varphi^\ast (J_F(y_k))))_k$ contains a convergent subsequence in $G^{\ast}$, that is $\displaystyle \lim_{j\to \infty} \overline{P}(\varphi^\ast (J_F(y_{k_j}))) = g^{\ast}\in G^*$.
    Finally, since $\overline{\varphi_G}(P)(y_k) = \overline{P}(\varphi^\ast (J_F(y_k)))|_{G}$ for every $k \in \N$, we conclude that 
    $$ \displaystyle \lim_{j \to \infty} \overline{\varphi_G}(P)(y_{k_j}) =   \displaystyle \lim_{j \to \infty} \overline{P}(\varphi^\ast (J_F(y_{k_j})))|_{G} = g^{\ast} \in G^\ast, $$
    proving that $\overline{\varphi_G}(P)$ is compact. In particular, we obtain that 
   $\overline{\varphi_G}\colon
      \mathcal{P}_{\mathcal{K}}^r(^n E; G^\ast) \to
      \mathcal{P}_{\mathcal{K}}^r(^n F; G^\ast)$
    is a well-defined lattice homomorphism. To conclude that
    $\overline{\varphi_G}$
      is an isomorphism, just note that the same argument used above shows that
    $\overline{\varphi_G^{-1}}(Q)$ is a compact polynomial for every $Q \in \mathcal{P}_{\mathcal{K}}^+(^n F; G^\ast)$.
\end{proof}

The following holds by applying the same argument used above replacing \cite[Example 2.3]{botmirru} with \cite[Example 2.4]{botmirru}.

\begin{proposition} \label{cor2}
  Let $E, F, G$ be Banach lattices. Suppose that there exists a lattice isomorphism $\varphi\colon E^\ast \to F^\ast$ and that both $E^*$ and $G^\ast$ have
  order continuous norms.
Then, $\mathcal{P}_{\mathcal{W}}^r(^n E; G^\ast)$ and $\mathcal{P}_{\mathcal{W}}^r(^n F; G^\ast)$ are lattice isomorphic.
\end{proposition}

 Given Banach lattices $E$, $F$, and $n\in \mathbb{N}$,  we
  say that an $n$-homogeneous polynomial $P\in \mathcal{P}^r(^n E; F)$  is orthogonally additive
  if
$P(x+y)=P(x)+P(y)$ whenever $|x|\wedge|y|=0$; or equivalently if $T_P$ is orthosymmetric, that is $T_P(x_1,\dots,x_n)=0$ whenever
some pair $x_i$, $x_j$  of the arguments are disjoint, i.e. $|x_i|\wedge |x_j| = 0$ (see \cite[Lemma 4.1]{bubuskes}).
In the case where $F$ is Dedekind complete, the Banach lattice of all regular polynomials $P\colon E \to F$ for which $|P|$ is orthogonally additive is
denoted by $\mathcal{P}_o^r (^n E; F)$ (see \cite[Section 5]{bubuskes}).

\begin{theorem} If $E^\ast$ and $F^\ast$ are lattice isomorphic with at least
  one (and hence the other one) having an order continuous norm, then
  $\mathcal{P}^r_o(^n E; G^*)$ and $\mathcal{P}^r_o(^n F; G^*)$
  are lattice isomorphic for every $n\in\mathbb{N}$ and every Banach lattice $G$.
\end{theorem}

\begin{proof}
  Let $\varphi\colon E^* \to F^*$ be the lattice isomorphism given by the
  assumption. By \cite[Theorem 2.20]{alip} we get that $\varphi^*$ is interval preserving. Since $\varphi$ is surjective, $\varphi^*$ is also injective.
  Thus, $\varphi^*$ is a lattice homomorphism by
    \cite[Proposition 2.1]{vanamstel}.
  Letting $\overline{\varphi_{G}}\colon \mathcal{P}^r(^n E; G^\ast) \to
  \mathcal{P}^r(^n F; G^{\ast})$ be the lattice isomorphism defined in the proof of Theorem \ref{teorema3}, we claim that $\overline{\varphi_G}(P)$ is an orthogonally additive polynomial for every $P \in \mathcal{P}_{o}^+(^n E; G^\ast)$. Indeed, given $y_1, \dots, y_n \in F$ with $|y_i| \wedge |y_j| = 0$ for some $i \neq j$, we have that
  $|\varphi^*(J_F(y_i))| \wedge |\varphi^* (J_F(y_j)) |= 0$, because both
  $\varphi^*$ and $J_F$ preserve disjointness. Since the
  Arens-extension of an order bounded orthosymmetric $n$-linear map is also orthosymmetric (see \cite[Corollary 5.18]{roberts}), we conclude that
    $$ T_{\overline{\varphi_G}(P)}(y_1, \dots, y_n) = \overline{T_P}(\varphi^*(J_F(y_1)), \dots, \varphi^*(J_F(y_n))) = 0, $$
    proving that $T_{\overline{\varphi_G}(P)}$ is orthosymmetric. Thus, $\overline{\varphi_G}(P) \in \mathcal{P}_{o}^+(^n F; G^*)$ as we claimed. Analogously, we can prove that $\overline{\varphi_G^{-1}}(Q) \in \mathcal{P}_{o}^+
      (^n E; G^*)$
    for each $Q \in
      \mathcal{P}_{o}^+(^n F; G^*)$, and therefore
   $\overline{\varphi_G}\colon \mathcal{P}_{o}^r (^n E; G^*) \to \mathcal{P}_{o}^r (^n F; G^*)$
    is a lattice isomorphism.
\end{proof}

\section{Regular nuclear polynomials}

Let $E$ and $F$ be two Banach lattices with $F$ being Dedekind complete.
Following \cite[p. 725]{blancolal}, a bounded linear operator $T\colon E
\to F$
is said to be a \textit{regular nuclear operator} if it belongs to the image of the map 
$$\overline{\Phi}\colon E^\ast \,\widehat{\otimes}_{|\pi|} F \to \mathcal{L}^r(E; F), \quad \overline{\Phi}(x^* \otimes y)(x) = x^*(x) y.$$

To present a polynomial version of such class, fix $n \in \N$ and let $\phi_1, \dots, \phi_n \in (E^*)^+$ be given.
  If $A: E \times \overset{n}{\cdots} \times E \to \R$ is the $n$-linear operator defined by 
$$ A(x_1, \dots, x_n) = \phi_1(x_1) \cdots \phi_n(x_n), $$
its symmetrization is given by
$$ A^s(x_1, \dots, x_n) = \frac{1}{n!} \sum_{\sigma \in S_n} A(x_{\sigma(1)}, \dots, x_{\sigma(n)}) = \frac{1}{n!} \sum_{\sigma \in S_n} \phi_1(x_{\sigma(1)}) \cdots \phi_n(x_{\sigma(n)}). $$
Since $\phi_1, \dots, \phi_n \geq 0$, we get that $A^s \geq 0$. Thus $P_{\phi_1, \dots, \phi_n}(x):= A(x, \dots, x) = \phi_1(x) \cdots \phi_n(x)$ defines a positive $n$-homogeneous polynomial on $E$, and so we can define a positive symmetric $n$-linear map
$$ (\phi_1, \dots, \phi_n) \in E^* \times \cdots \times E^* \mapsto P_{\phi_1, \dots, \phi_n} \in \mathcal{P}^r(^nE).$$
Its range is contained in $\mathcal{P}^r(^nE)$. For
instance, taking $n = 2$, 
$$ \phi_1 \cdot \phi_2 = (\phi_1^+ - \phi_1^-)(\phi_2^+ - \phi_2^-) = (\phi_1^+\phi_2^+ + \phi_1^- \phi_2^-) - (\phi_1^+\phi_2^- + \phi_1^- \phi_2^+). $$
Now, we can consider the positive bilinear map 
$B \colon \widehat{\otimes}_{n, s, |\pi|}
E^\ast \times F  \to \mathcal{P}^r(^nE ; F)$ given by $B(\theta_n(x^*), y)(x)  = (x^*(x))^n y$ for all $0 \leq x^* \in E^*$, $y \in F^+$, $x\in E^+$. Hence, $B$ induces a
positive linear operator $\overline{\Phi}_n\colon \widehat{\otimes}_{n, s, |
  \pi|} E^\ast \,\widehat{\otimes}_{|\pi|} F \to \mathcal{P}^r(^nE; F)$ such
that
$$ \overline{\Phi}_n(\theta_n(x^*) \otimes y)(x) = (x^*(x))^n y \quad
\text{for all $0 \leq x^* \in E^*$, $y \in F^+$ and $x \in E^+$.} $$
The elements of $\text{Im }  \overline{\Phi}_n := \mathcal{P}_{\mathcal{N}}^r(^n E; F)$  are called
\textit{regular nuclear $n$-homogeneous polynomials}.
As in \cite{blancolal}, $\mathcal{P}_{\mathcal{N}}^r(^n E; F)$ is a Banach lattice endowed with the norm $\norma{\cdot}_{\mathcal{N}, r}$ induced by the quotient norm on 
$\widehat{\otimes}_{n, s, |\pi|} E^\ast \,\widehat{\otimes}_{|\pi|} F / \ker \overline{\Phi}_n$. 
The main objective of this section is to prove that
if $G$ is a Dedekind
complete Banach lattice with the lattice approximation property,
and both
$E^*$ and $F^*$ have the lattice approximation property and are
    lattice isomorphic, then $\mathcal{P}_{\mathcal{N}}^r(^n E; G)$ and $\mathcal{P}_{\mathcal{N}}^r(^n F; G)$ are lattice isomorphic. To do this, we recall from \cite{blancolal} that a \textit{principal pair} of
a Banach lattice
$E$ is an ordered pair $(G, g)$, in which $G$ is a finite-dimensional vector sublattice of $E$ and $g$ is a strong order unit of $G$, i.e. for each $y \in G$, exists $\lambda > 0$ with $|y| \leq \lambda g$.
A Banach lattice $E$ is said to
  have the \textit{lattice approximation property} (LAP in
  short)  if for every $\varepsilon > 0$ and every sequence $(E_n, e_n)$ of
  principal pairs of $E$ such that $\|e_n\|\to 0$ there
  exists a finite-rank operator $T\colon E \to E$ such
  that $\norma{|(T - id)_{E_n}|(e_n)} < \varepsilon$ for every $n \in \N$.

  The following lemma, which can be seen as a generalization
  of \cite[Lemma 3.5]{blancolal}, will be necessary to achieve our objective.

\begin{lemma} \label{lema31}
    Let $E$ and $F$ be Dedekind complete Banach lattices and $w \in \widehat{\otimes}_{n, s, |\pi|} E \, \widehat{\otimes}_{|\pi|} \, F$. For each $0 < \varepsilon \leq 1$, there exist finite-dimensional vector sublattices $E_0 \subset E$ and $F_0 \subset F$ and a vector $\widetilde{w} \in \otimes_{n,s} E_0 \, \otimes F_0$ such that $\norma{w - \tilde{w}}_{ |\pi|} \leq \varepsilon$ and $\norma{\tilde{w}}_{\otimes_{n,s}E_0 \widehat{\otimes}_{|\pi|} F_0} \leq (1+\varepsilon)^n \norma{\tilde{w}}_{s, |\pi|} + 2^n \varepsilon$
\end{lemma}

\begin{proof}
    Since $\widehat{\otimes}_{n, s, |\pi|} E \otimes \, F$ is dense in $\widehat{\otimes}_{n, s, |\pi|} E \, \widehat{\otimes}_{|\pi|} \, F$ there exists $w_0 = \displaystyle\sum_{j=1}^m u_j \otimes y_j \in \widehat{\otimes}_{n, s, |\pi|} E \otimes \, F$ such that
    $ \norma{w - w_0}_{|\pi|} \leq \dfrac{\varepsilon}{4}$.
    Without loss of generality, we assume that $u_j \neq 0$ and $y_j \neq 0$ for every $j =1, \dots, m$.
    Now, from the density of ${\otimes}_{n, s} E$ in $\widehat{\otimes}_{n, s, |\pi|} E$, for each $j=1, \dots, m$, there exists
      $u_{j,0} := \displaystyle \sum_{i=1}^{m_j} \lambda_{i,j} \theta_n(x_i^j)$
      such that $\norma{u_j - u_{j,0}}_{s, |\pi|} \leq \dfrac{\varepsilon}{4m \max \{ \norma{y_1}, \dots, \norma{y_m} \}}$. Thus
    $u: = \sum_{j=1}^m \sum_{i=1}^{m_j} \lambda_{i,j} \theta_n(x_i^j) \otimes y_j \in
    {\otimes}_{n, s} E \otimes F$ satisfies
    \begin{align*}
        \norma{w - u}_{|\pi|} & \leq \norma{w - w_0}_{|\pi|} + \norma{w_0 - u}_{|\pi|} \\
        & 
        \leq \frac{\varepsilon}{4} + \norma{\sum_{j=1}^m (u_j \otimes y_j -
          \sum_{i=1}^{m_j} \lambda_{i,j}\theta_n(x_i^j) \otimes y_j)}_{|\pi|} \\
        & = \frac{\varepsilon}{4} +  \norma{\sum_{j=1}^m (u_j - \sum_{i=1}^{m_j} \
          \lambda_{i,j} \theta_n(x_i^j)) \otimes y_j}_{|\pi|} \\
        & \leq \frac{\varepsilon}{4} + \sum_{j=1}^m \norma{u_j - \sum_{i=1}^{m_j}
          \lambda_{i,j} \theta_n(x_i^j)}_{s, |\pi|} \norma{y_j} \leq \frac{\varepsilon}{2}.
    \end{align*}
Also,  by the definition of the positive projective tensor norm, there exist $z_1, \dots, z_p \in (\widehat{\otimes}_{n, s, |\pi|} E)^+$ and $v_1, \dots, v_p \in F^+$ such that
    $$ |u| \leq \sum_{i=1}^p z_i \otimes v_i \, \text{ and } \,
    \sum_{i=1}^p \norma{z_i}_{s, |\pi|} \norma{v_i} \leq \norma{u}_{|\pi|} + \frac{\varepsilon}{4}. $$
    Without loss of generality, we assume that $v_i \neq 0$ for every $i = 1, \dots, p$.
    For each $i=1, \dots, p$, we apply the definition of the positive projective symmetric tensor norm to obtain $z_i^1, \dots, z_i^{p_i} \in E^+$ such that
    $$ z_i = |z_i| \leq \sum_{k=1}^{p_i} \theta_n(z_i^k) \, \text{ and } \, \sum_{k=1}^{p_i} \norma{z_i^k}^n \leq \norma{z_i}_{s, |\pi|} + \frac{\varepsilon}{4p \max \{ \norma{v_1}, \dots, \norma{v_p} \}}. $$
    Therefore, 
    $$ |u| \leq \sum_{i=1}^p z_i \otimes v_i \leq \sum_{i=1}^p (\sum_{k=1}^{p_i} \theta_n(z_i^k)) \otimes v_i = \sum_{i=1}^p \sum_{k=1}^{p_i} \theta_n(z_i^k) \otimes v_i $$
    and
    \begin{align*}
        \sum_{i=1}^p \sum_{k=1}^{p_i} \norma{z_i^k}^n \norma{v_i} & = \sum_{i=1}^p \left( \sum_{k=1}^{p_i} \norma{z_i^k}^n \right )\norma{v_i} \leq \sum_{i=1}^p \left (\norma{z_i}_{s, |\pi|} + \frac{\varepsilon}{4p \max \{ \norma{v_1}, \dots, \norma{v_p} \}} \right ) \norma{v_i} \\
        & = \sum_{i=1}^p \norma{z_i}_{s, |\pi|}\norma{v_i} + \frac{\varepsilon}{4}  \leq \norma{u}_{|\pi|} + \frac{\varepsilon}{4} + \frac{\varepsilon}{4} = \norma{u}_{|\pi|} + \frac{\varepsilon}{2}.
    \end{align*}    
    Let 
    $$E_0 := {\rm span}\{ x_1^{j}, \dots, x_{m_j}^{j}, \, z_i^1, \dots, z_i^{p_i} \, : j =1, \dots, m, \, i = 1, \dots, p  \} \subset E$$ 
    and 
    $$F_0 :=  {\rm span} \{ y_1, \dots, y_m, \, v_1, \dots, v_p \} \subset F.$$
    Since $E$ and $F$ are Dedekind complete, by \cite[Proposition 2.1]{blancolal}, for each $\delta > 0$, there are finite-dimensional vector sublattices $E_\delta \subset E$, $F_\delta \subset F$, and linear maps 
    $$ U_\delta: E_0 \to E_\delta, \, V_\delta: F_0 \to F_\delta $$  
    such that 
    $$ \norma{U_\delta(x) - x} \leq \delta \norma{x} \, \text{for every
      $x \in E_0$} \, \text{ and } \,
    {\norma{V_\delta(y) - y}\le \delta\norma{y}}
    \, \text{for every $y \in F_0$}. $$
    For each $0 < \delta < 1 $, set $v^\delta = \sum_{i=1}^p \sum_{k=1}^{p_i} \theta_n(U_{\delta}(z_i^k)) \otimes V_{\delta}(v_i)$. We claim that $v^\delta \longrightarrow v := \sum_{i=1}^p \sum_{k=1}^{p_i} \theta_n(z_i^k) \otimes v_i$ as $\delta \longrightarrow 0$. We have
    \begin{align*}
        \norma{v - v^\delta}_{|\pi|} & \leq \sum_{i=1}^p \sum_{k=1}^{p_i} \norma{\theta_n(U_{\delta}(z_i^k)) \otimes V_{\delta}(v_i) - \theta_n(z_i^k) \otimes v_i}_{|\pi|} \\
        & \leq \sum_{i=1}^p \sum_{k=1}^{p_i}[ \norma{\theta_n(U_\delta(z_i^k))}_{s, |\pi|} \norma{V_\delta(v_i) - v_i} +  \norma{\theta_n(U_\delta(z_i^k)) - \theta_n(z_i^k)}_{s, |\pi|} \norma{v_i}] \\
        & \leq \sum_{i=1}^p \sum_{k=1}^{p_i}[ \norma{U_\delta(z_i^k)}^n \norma{V_\delta(v_i) - v_i} +  n \, \max \{ \| U_\delta(z_i^k) \|, \norma{z_i^k} \}^{n-1} \norma{U_\delta(z_i^k) - z_i^k} \norma{v_i}] \\
        & \leq \sum_{i=1}^p \sum_{k=1}^{p_i} [M^{n} \delta \norma{v_i} + n  M^{n} \delta \norma{z_i^k}] \\
        & \leq \delta  \sum_{i=1}^p \sum_{k=1}^{p_i} [M^{n+1} + n M^{n+1}],
    \end{align*}
    provided $\delta\le 1$,
    where $M = \displaystyle \max_{i=1, \dots, p} \{ \norma{U_\delta(z_i^1)}, \dots, \norma{U_\delta(z_i^{p_i})}, \norma{z_i^1}, \dots, \norma{z_i^{p_i}}, \norma{v_i}  \}$. Letting $\delta \to 0$, we obtain that $\norma{v -
      v^\delta}_{|\pi|} \longrightarrow 0$. Analogously, we have that 
    $$ u^\delta := \sum_{j=1}^m \sum_{i=1}^{m_j} \lambda_{i,j}\theta_n(U_\delta(x_i^j)) \otimes V_\delta(y_j) \longrightarrow u. $$
     The continuity of the lattice operations yields that 
       $$ w^\delta := v^\delta \wedge ((-v^\delta) \vee u^\delta) \to 
v \wedge ((-v) \vee u) = u \quad \text{as } \delta \to 0. $$

Now, we choose $0 < \delta_0 < (1+\varepsilon)^{n/(n+1)} - 1$ such that
$ \|u - w^{\delta_0} \|_{|\pi|} \le \frac{\varepsilon}{2}$.
We will prove that $E_{\delta_0}$, $F_{\delta_0}$ and
$\widetilde{w} := w^{\delta_0}$
satisfy the hypothesis of the lemma. First, we note that since
$E_{\delta_0}$ and $F_{\delta_0}$ are finite-dimensional vector lattices, $\otimes_{n,s} E_{\delta_0} \otimes F_{\delta_0}$ is finite-dimensional vector lattice,
and so
$\tilde{w} \in \otimes_{n,s} E_{\delta_0} \otimes F_{\delta_0}$.
Secondly, we have that
$$ \|w - \tilde{w}\|_{|\pi|} 
\le \|w - u\|_{|\pi|} + \|u - \tilde{w}\|_{|\pi|} 
\le \varepsilon.$$
Moreover, since
$$ |\widetilde{w}| \leq |v^{\delta_0}| \leq \sum_{i=1}^p \sum_{k=1}^{p_i} |\theta_n(U_{\delta_0}(z_i^k))| \otimes |V_{\delta_0}(v_i)| =  \sum_{i=1}^p \sum_{k=1}^{p_i} \theta_n(|U_{\delta_0}(z_i^k)|) \otimes |V_{\delta_0}(v_i)|, $$
we get
\begin{align*}
    \norma{\widetilde{w}}_{\otimes_{n,s}E_{\delta_0} \widehat{\otimes}_{|\pi|} F_{\delta_0}} &\leq \sum_{i=1}^p \sum_{k=1}^{p_i} \norma{\theta_n(U_{\delta_0}(z_i^k))}_{s, |\pi|} \norma{V_{\delta_0}(v_i)} \\
    & = \sum_{i=1}^p \sum_{k=1}^{p_i} \norma{U_{\delta_0}(z_i^k)}^n \norma{V_{\delta_0}(v_i)} \leq \sum_{i=1}^p \sum_{k=1}^{p_i} (1+\delta_0)^{n+1} \norma{z_i^k}^n \norma{v_i}  \\
    & \leq (1+\delta_0)^{n+1}( \norma{u}_{|\pi|} + \frac{\varepsilon}{2}) \leq (1+\varepsilon)^{n} (\norma{\widetilde{w}}_{|\pi|} + \varepsilon) \leq (1+\varepsilon)^{n} \norma{\widetilde{w}}_{|\pi|} + 2^n \varepsilon,
\end{align*}
and we are done.
\end{proof}

We will also need the following technical lemma.

\begin{lemma} \label{lemanuclear}
  Let $E$ and $F$ be two Dedekind complete Banach lattices.
  For every $w \in  \widehat{\otimes}_{n, s, |\pi|} E \,\widehat{\otimes}_{|\pi|} F$, there exist sequences $(E_k)_k$ and $(F_k)_k$ of finite-dimensional vector sublattices of $E$ and $F$, respectively, and a sequence $(w_k)_k$ in $\widehat{\otimes}_{n, s, |\pi|} E \,\widehat{\otimes}_{|\pi|} F$ such that $\displaystyle
  w=\sum_{k=1}^\infty w_k$,
  $w_k \in \otimes_{n, s} E_k \otimes F_k$
  and $\norma{w_k}_{\otimes_{n, s} E_k  \,\widehat{\otimes}_{|\pi|} F_k} \leq 2^n \norma{w_k}_{|\pi|} + 2^{n-k}$ for every $k \in \N$.  Moreover, $\displaystyle \norma{w_k
  }_{\otimes_{n, s} E_k  \,\widehat{\otimes}_{|\pi|} F_k} \leq 2^{n-k+2}$ holds for every
  $k \in \N$ whenever $\norma{w}_{|\pi|} = 1$.
\end{lemma}

\begin{proof}
      Let $w \in \widehat{\otimes}_{n, s, |\pi|} E  \,\widehat{\otimes}_{|\pi|} F$ be given. 
      By Lemma \ref{lema31}, there exist finite-dimensional vector sublattices
$E_1 \subset E$ and $F_1 \subset F$, and a vector $w_1 \in \otimes_{n, s} E_1  \,
    \otimes F_1$ such that $\norma{w - w_1}_{|\pi|} \leq 2^{-1}$ and  
  $$\norma{w_1}_{\otimes_{n, s} E_1 \,\widehat{\otimes}_{|\pi|} F_1} \leq \left(1+
  \frac{1}{2}\right)^n\norma{w_1}_{|\pi|} + 2^n \dfrac{1}{2} \leq 2^n
  \norma{w_1}_{|\pi|} + 2^{n-1}.$$
Using Lemma \ref{lema31} again, there exist finite-dimensional vector sublattices $E_2 \subset E$
and $F_2 \subset F$,
and a vector $w_2 \in \otimes_{n, s} E_2 \,\otimes F_2$ such that $\norma{(w - w_1) - w_2}_{|\pi|} \leq 2^{-2}$ and
$$\norma{w_2}_{\otimes_{n, s} E_2 \,\widehat{\otimes}_{|\pi|} F_2} \leq \left(1+
       \frac{1}{2^2}\right)^n\norma{w_2}_{|\pi|} + 2^n \frac{1}{2^2} \leq 2^n
       \norma{w_2}_{|\pi|} + 2^{n-2}. $$
 Proceeding inductively, we may construct sequences $(E_k)_k$ and $(F_k)_k$ of finite-dimensional vector sublattices of $E$ and $F$, respectively, and a sequence $(w_k)_k \subset \widehat{\otimes}_{n, s, |\pi|} E \,\widehat{\otimes}_{|\pi|} F$ such that $w_k \in \otimes_{n, s} E_k \,
 \otimes F_k$, $\norma{(w - w_1 - \cdots - w_{k-1}) - w_k}_{|\pi|} \leq 2^{-k}$ and $$\norma{w_k}_{\otimes_{n, s} E_k \,\widehat{\otimes}_{|\pi|} F_k} \leq 2^n\norma{w_k}_{|\pi|} + 2^{n-k} \, \text{ for every $k > 1$.}$$
 
 Thus, $\left\|w - \sum_{i=1}^k  w_i\right\|_{|\pi|} \leq 2^{-k} $    holds for
 every $k \in \N$. By letting $k \to \infty$, we get that $w=\displaystyle\sum_{k=1}^\infty w_k$. Moreover, for every $k \geq 2$,
\begin{align*}
  \norma{w_k}_{|\pi|} & = \left\| w - \sum_{i=1}^k w_i - (w - \sum_{j=1}^{k-1}
  w_j) \right\|_{|\pi|} \leq \left\| w - \sum_{i=1}^k w_i\right\| + \left\| w -
  \sum_{i=1}^{k-1} w_j\right\| \\
    & < 2^{-k} + 2^{-(k-1)} = 3 \cdot 2^{-k},
\end{align*}
and so
\begin{align*}
  \norma{w_k}_{\otimes_{n,s} E_k \widehat{\otimes}_{|\pi|} F_k} &
  \le 2^n \norma{w_k}_{|\pi|} + 2^{n-k} \leq 2^n
    3 \cdot 2^{-k} + 2^{n-k} \\
    & = 4 \cdot 2^{n-k} = 2^{n-k +2}. 
\end{align*}
Assuming $\norma{w}_{|\pi|} = 1$, for $k =1$, we have
\begin{align*}
    \norma{w_1}_{\otimes_{n,s} E_1 \otimes_{|\pi|} F_1} & \leq 2^n \norma{w_1}_{|\pi|} + 2^{n-1} \leq 2^n (\norma{w_1 - w}_{|\pi|} + \norma{w}_{|\pi|}) + 2^{n-1} \\
    & < 2^n (2^{-1} + 1) + 2^{n-1} = 2^n \frac{3}{2} + 2^{n-1} = 2^{n-1 + 2},
\end{align*}
and we are done. 
\end{proof}

\begin{lemma}\label{lemaestimate}
  Let $H$ be a Dedekind complete Banach lattice, let $G$ be a finite-dimensional
  vector sublattice of $H$, and let $S: H \to H$ be a regular operator.
Denote by
$D_G: \widehat{\otimes}_{n,s,|\pi|} G \to \widehat{\otimes}_{n,s,|\pi|} H $
the linearization of the regular $n$-homogeneous polynomial
$$ x \in G \mapsto \theta_n(x) - \theta_n(Sx) \in \widehat{\otimes}_{n,s, |\pi|} H. $$
Then $D_G$ has modulus and, for every $g \in G^+$, 
$$ |D_G|(\theta_n(g)) \leq \theta_n(g + |(S-id)_G|(g) ) - \theta_n(g). $$
\end{lemma}

\begin{proof} Considering
  the symmetric $n$-linear operator
$$ B(x_1, \dots, x_n) = Sx_1 \otimes_s \cdots \otimes_s Sx_n - x_1 \otimes_s \cdots \otimes_s x_n, \qquad x_1, \dots, x_n \in G,$$
we have
$ \displaystyle B(x_1, \dots, x_n) = \sum_{\emptyset \neq A \subset \{1 , \dots, n\}} T_1^A(x_1) \otimes_s \cdots \otimes_s T_n^A(x_n), $
where 
$$\displaystyle  T_r^A(z) = 
\begin{cases}
(S-id)(z), & r\in A\\
z, & r\notin A
\end{cases}, \qquad z \in G. $$
Hence
\begin{align*}
  |B(x_1, \dots, x_n)| & \leq \sum_{\emptyset \neq A \subset \{1 , \dots, n\}}
      |T_1^A|(|x_1|) \otimes_s \cdots \otimes_s |T_n^A|(|x_n|) \\
                       & \leq
                         \sum_{\emptyset \neq A \subset \{1 , \dots, n\}} R_1^A(|x_1|) \otimes_s \cdots \otimes_s R_n^A(|x_n|),
\end{align*}
where
$$ R_r^A(z) = \begin{cases}
|(S-id)_{G}|(z), & r\in A\\
z, & r\notin A
\end{cases}, \qquad z \in G. $$

Define the positive symmetric $n$-linear operator 
$$ C(x_1, \dots, x_n) := \sum_{\emptyset\neq A \subset \{1, \dots, n\}} R_1^A(x_1) \otimes_s \cdots \otimes_s R_n^A(x_n), \qquad x_1, \dots, x_n \in G.$$
The preceding inequality shows that
$- C \leq B \leq C$ holds, hence $C+B$ and $C-B$ are positive symmetric $n$-linear operators. By \cite[Proposition 3.4]{bubuskes}, their linearizations are positive, and therefore $-C^\otimes \leq B^\otimes \leq C^\otimes$. Since $G$ is finite-dimensional, $\widehat{\otimes}_{n, s, |\pi|} G$ is also finite-dimensional, and so the modulus of $B^\otimes$ exists with $|B^\otimes| \leq C^\otimes$. Moreover, 
\begin{align*}
    B^\otimes (\theta_n(g)) & = B(g, \dots, g) = Sg \otimes_s \cdots \otimes_s Sg - g \otimes_s \cdots \otimes_s g = (\otimes^n S - id)(\theta_n(g)),
\end{align*}
and so $D_G = -B^\otimes$. Thus, $|D_G|$ exists and, for each $g \in G^+$,
\begin{align*}
    |D_G|(\theta_n(g)) & = |id - \otimes^n S| (\theta_n(g)) = |B^\otimes|(\theta_n(g)) \\
    & \leq C^\otimes(\theta_n(g)) = C(g,\dots, g) \\
 & = \theta_n\left(g+\left|(S-id)_G\right|(g)\right)
 -\theta_n(g).
\end{align*}
\end{proof}

It was proved in \cite{blancolal} that, whenever $F$ is a Dedekind complete Banach lattice with the LAP, then $\overline{\Phi}\colon E^\ast \,\widehat{
  \otimes}_{|\pi|} F \to \mathcal{N}^r(E; F)$ is one-to-one, and hence an isometry, for every Banach lattice $E$. The same argument used in Blanco's proof can be adapted to obtain a polynomial version of this result that will be presented in our next theorem. First, we recall that  $\mathcal{P}_f(^n E; F)$ denotes the
vector space of all finite-type $n$-homogeneous polynomials from $E$ into
$F$.

\begin{theorem} \label{teonuclear}
    Let $E$ and $F$ be Banach lattices with $F$ Dedekind complete. If $E^*$ and $F$ have the LAP, then  the map $\overline{\Phi}_n\colon\widehat{\otimes}_{n, s, |\pi|} E^\ast \,\widehat{\otimes}_{|\pi|} F \to \mathcal{P}_{\mathcal{N}}^r(^n E; F)$ is injective, and hence an isometry.
\end{theorem}

\begin{proof} Suppose for the sake of contradiction that $\overline{\Phi}_n$
  is not injective. Thus, there exists
  $w \in \widehat{\otimes}_{n, s, |\pi|} E^\ast \,\widehat{\otimes}_{|\pi|} F$ with $\norma{w}_{|\pi|} = 1$ such that $\overline{\Phi}_n(w) = 0$.  Since $E^*$ and $F$ are Dedekind complete,
  we can apply Lemma \ref{lemanuclear} to obtain sequences $(G_k)_k$ and $(F_k)_k$ of finite-dimensional vector sublattices of $E^*$ and $F$, respectively, and the corresponding sequence $(w_k)_k \subset \widehat{\otimes}_{n,s, |\pi|} E^* \widehat{\otimes}_{|\pi|} F$ satisfying $w_k \in \otimes_{n, s} G_k 
  \otimes F_k$ and $\norma{w_k}_{\otimes_{n, s} G_k \,\widehat{\otimes}_{|\pi|} F_k}
  \leq 2^n \norma{w_k}_{|\pi|} + 2^{n-k} \leq 2^{n-k+2}$ for every $k \in \N$. 
  Thus, there are positive elements
  $z_{k,1}, \dots, z_{k, m_k} \in \otimes_{n, s} G_k$
    and positive vectors $y_{k,1}, \dots, y_{k, m_k} \in F_k $ such that 
    $$|w_{k}| \leq \displaystyle \sum_{j=1}^{m_k} z_{k,j} \otimes y_{k,j} \quad \text{  
    and} \quad \displaystyle \sum_{j=1}^{m_k} \norma{z_{k,j}}_{s, |\pi|} \norma{ y_{k,j}} \leq {2^{n-k+2}}.$$
\textbf{Claim 1:} For each $k \in \N$, there are $0 \leq u_{k,1}, \dots, u_{k, m_k}$ $\in$ $\otimes_{n, s} G_k$ and $0 \leq v_{k,1}, \dots, v_{k, m_k} \in
     F_k $ such that
     \begin{equation} \label{eq1}
         |w_k| \leq \sum_{j=1}^{m_k} u_{k,j}
     \otimes v_{k,j}, \, \max_{1 \leq j \leq m_k}
     \norma{v_{k,j}} = \frac{1}{k}, \,  \sum_{k=1}^\infty \sum_{j=1}^{m_k} \norma{u_{k,j}}_{s, |\pi|} < \infty.
     \end{equation}     
     Indeed, first note that we can assume that $y_{k,j}
     \not= 0$ as if it is
     equal to  $0$, then $z_{k,j}\otimes y_{k,j}=0$. For each $k \in
     \N$ and each $j = 1, \dots, m_k$, we let 
     $$ u_{k,j} = k \norma{y_{k,j}} z_{k,j} \quad \text{and} \quad
     v_{k,j} = \frac{y_{k,j}}{k \norma{y_{k,j}}}.$$
     Thus, $0 \leq u_{k,1}, \dots, u_{k, m_k}$ $\in$ $\otimes_{n, s} G_k$ and $0 \leq v_{k,1}, \dots, v_{k, m_k} \in
     F_k $ satisfy
     $$ |w_k| \leq \sum_{j=1}^{m_k} z_{k,j} \otimes y_{k,j} =  \sum_{j=1}^{m_k} u_{k,j}
     \otimes v_{k,j}, \quad \max_{1 \leq j \leq m_k}
     \norma{v_{k,j}} = \frac{1}{k}, $$
and
$$ \sum_{k=1}^\infty \sum_{j=1}^{m_k} \norma{u_{k,j}}_{s, |\pi|} =
\sum_{k=1}^\infty \sum_{j=1}^{m_k} k \norma{z_{k,j}}_{s, |\pi|} \norma{ y_{k,j}} \leq  \sum_{k=1}^\infty k \cdot 2^{n-k+2} < \infty. $$

\noindent {\bf Claim 2:} For each $k \in \N$ and each $j = 1, \dots, m_k$, there are $g_{k,j}^1, \dots, g_{k,j}^{p_{k,j}} \in G_k^+$ and $\widetilde{v}_{k,j}^1, \dots, \widetilde{v}_{k,j}^{p_{k,j}} \in F_k^+$ such that $|w_k| \leq  \displaystyle \sum_{j=1}^{m_k} \sum_{i=1}^{p_{k,j}} \theta_n(g_{k,j}^i)
     \otimes \widetilde{v}_{k,j}^i,$
\begin{equation} \label{eq2}
    \max_{1 \leq j \leq m_k} \max_{1 \leq i \leq p_{k,j}}
     \norma{g_{k,j}^i} = \frac{1}{k} \, \text{ and } \, \sum_{k=1}^\infty \sum_{j=1}^{m_k} \sum_{i=1}^{p_{k,j}} \norma{\widetilde{v}_{k,j}^i} < \infty.
\end{equation}
First,
for each $k \in \N$ and each $j = 1, \dots, m_k$, there are $\varphi_{k,j}^1, \dots, \varphi_{k,j}^{p_{k,j}} \in G_k^+$ such that $z_{k,j} \leq \sum_{i=1}^{p_{k,j}} \theta_n(\varphi_{k,j}^i)$ and
$$ \sum_{i=1}^{p_{k,j}} \norma{\varphi_{k,j}}^n \leq \norma{z_{k,j}}_{s, |\pi|} + \frac{1}{m_k \displaystyle \max_{1 \leq j \leq m_k} \norma{y_{k,j}} \, 2^k}. $$
Next, we note that, as in Claim 1,  we can assume that
$\varphi_{k,j}^i\not=0$, and so we define
$$ g_{k,j}^i =\frac{\varphi_{k,j}^i}{k \norma{\varphi_{k,j}^i}} \quad \text{and} \quad \widetilde{v}_{k,j}^i := (k \norma{\varphi_{k,j}^i})^n y_{k,j}. $$
$0 \leq g_{k,j}^i \in G_k$ and $0 \leq \widetilde{v}_{k,1}
  ^i, \dots, \widetilde{v}_{k, m_k}^i \in
F_k $ satisfy
$$
|w_k| \leq \sum_{j=1}^{m_k} z_{k,j} \otimes y_{k,j}
  \leq
  \sum_{j=1}^{m_k} \sum_{i=1}^{p_{k,j}} \theta_n(g_{k,j}^i)
  \otimes  {\tilde v}^i_{k,j}, \\
\quad   \max_{1 \leq j \leq m_k} \max_{1 \leq i \leq p_{k,j}}
     \norma{g_{k,j}^i} = \frac{1}{k}, $$
and
\begin{align*}
    \sum_{k=1}^\infty \sum_{j=1}^{m_k} \sum_{i=1}^{p_{k,j}} \norma{\widetilde{v}_{k,j}^i} & =
\sum_{k=1}^\infty \sum_{j=1}^{m_k} \sum_{i=1}^{p_{k,j}} k^n \norma{\varphi_{k,j}^i}^n \norma{y_{k,j}} \\
& \leq \sum_{k=1}^\infty \sum_{j=1}^{m_k} k^n\left(\norma{z_{k,j}}_{s, |\pi|} + \frac{1}{m_k \max_{1 \leq j \leq m_k} \norma{y_{k,j}} 2^k}\right) \norma{y_{k,j}} \\
& \leq \sum_{k=1}^\infty \sum_{j=1}^{m_k} k^n \norma{z_{k,j}}_{s, |\pi|} \norma{y_{k,j}} + \sum_{k=1}^\infty \frac{k^n}{2^k} \\
& \leq  \sum_{k=1}^\infty k^n \cdot 2^{n-k+2} + \sum_{k=1}^\infty
\frac{k^n}{2^k} < \infty.
\end{align*}

For each $k, j, i$, define
 $$G_{k,j}^i := \conj{g \in G_k}{|g| \leq m g_{k,j}^i  \,\text{ for some } m \in \N}$$ and
$$F_{k,j} = \conj{y \in F_k}{|y| \leq m \, v_{k,j} \,\text{ for some } m \in \N}.$$
Ordering $(v_{k,j})$ by 
$$ v_{1, 1}, \dots, v_{1, m_1}, v_{2,1} \dots, v_{2,m_2}, \dots $$
  and $(g_{k,j}^i)$ by
  $$ g_{1,1}^1, \dots, g_{1,1}^{p_{1,1}}, g_{1,2}^1, \dots, g_{1,2}^{p_{1,2}}, \dots, g_{1, m_1}^1, \dots, g_{1,m_1}^{p_{1, m_1}}, \dots  $$
we get that $(G_{k,j}^i, g_{k,j}^i)$ and
$(F_{k,j},v_{k,j})$
are sequences of principal pairs of $E^*$ and $F$, respectively. Let $0 <
\varepsilon < 1 $. As $E^*$ and $F$ have the LAP, there are finite-rank operators $S \colon E^* \to E^*$ and $T\colon F \to F$ such that 
$$ \norma{|(S - id)_{G_{k,j}^i}|(g_{k,j}^i)} < \varepsilon \, \text{ and } \,   \norma{|(T-id)_{F_{k,j}}|(v_{k,j})} < \varepsilon$$ for all $k,j, i$.

Defining $\Psi\colon \widehat{\otimes}_{n, s, |\pi|} E^\ast \,\widehat{\otimes}_{|\pi|} F \to \widehat{\otimes}_{n, s, |\pi|} E^\ast \,\widehat{\otimes}_{|\pi|} F$ by $\Psi(\theta_n(x^*) \otimes y) = \theta_n(S(x^*)) \otimes T(y)$ for all $x^* \in  E^\ast$ and $y \in F$, and extending by linearity and
  continuity,
    we get that 
$\Psi(w_k) \in \otimes_{n, s} \, S(E^*) \otimes T(F)$
    for every $k \in \N$. To see that $\Psi$ is well-defined, we note that if $\sum_{i=1}^k \theta_n(x_i^*) \otimes y_i = \sum_{j=1}^l \theta_n(z_j^*) \otimes w_j$, then
\begin{align*}
    \sum_{i=1}^k \theta_n(Sx_i^*) \otimes Ty_i & = \sum_{i=1}^k\varphi_1^\otimes (\theta_n(x_i^*)) \otimes Ty_i  = \varphi_1^\otimes \otimes T (\sum_{i=1}^k \theta_n(x_i^*) \otimes y_i) \\
    & = \varphi_1^\otimes \otimes T(\sum_{j=1}^l \theta_n(z_j^*) \otimes w_j) = \sum_{j=1}^l \varphi_1^\otimes (\theta_n(z_j^*)) \otimes Tw_j \\
    & = \sum_{j=1}^l \theta_n(Sz_j^*) \otimes Tw_j,
\end{align*}
where $\varphi_1^\otimes$ is the linearization of the regular symmetric $n$-linear map $\varphi_1: E^* \times \cdots \times  E^* \to \widehat{\otimes}_{n, s, |\pi|} E^*$ given by $\varphi_1(x_1^*, \dots, x_n^*) = Sx_1^* \otimes_s \cdots \otimes_s Sx_n^*$.
     
     Moreover, as $\otimes_{n, s}
    S(E^*) \otimes T(F)$ is closed in $\widehat{\otimes}_{n, s, |\pi|} E^\ast \,\widehat{\otimes}_{|\pi|} F$ and $\Psi(w) = \displaystyle \sum_{k=1}^\infty
    \Psi(w_k)$, we have that $\Psi(w) \in \otimes_{n, s} S(E^*) \otimes T(F)$,
    which yields that $\overline{\Phi}_n(\Psi(w)) = \Phi_n(\Psi(w))$, where $\Phi_n$ is the following isomorphism of vector spaces:
      $$ \Phi_n\colon \otimes_{n, s} E^\ast \otimes F \to
      \mathcal{P}_f(^n E; F), \quad \Phi_n(\theta_n(x^*) \otimes y) = (x^*(x))^n y. $$

For each $P\in\mathcal{P}_{\mathcal N}(^n E;F)$, let
$\overline{P}:E^{**}\to F$ denote its Arens extension, which is
$F$-valued in this case (see \cite[Proposition 1.4]{carando}). If
$P=\displaystyle \sum_{i=1}^\infty (x_i^*(\cdot))^n y_i$
is any nuclear representation of $P$, then
$\overline{P}(x^{**}) = \sum_{i=1}^\infty
(x^{**}(x_i^*))^n y_i$ for every $x^{**} \in E^{**}$. We can now define  $\Gamma : \mathcal{P}_{\cal N}(^n E; F) \to \mathcal{P}_{\cal N}(^n E; F)$ by
$$\Gamma(P)(x):= T(\overline{P} (S^*J_E(x))), \qquad x \in E,$$
where $J_E:E\to E^{**}$ is the canonical embedding. Let us see that this definition does not depend on the choice of a nuclear representation of $P$. Indeed, if 
$P=\displaystyle \sum_{i=1}^\infty (x_i^*(\cdot))^n y_i,$
then, for every $x\in E$,
$$ \Gamma(P)(x) = T \left ( \sum_{i=1}^\infty (S^*(J_E(x)) (x_i^*) )^n y_i \right ) = \sum_{i=1}^\infty ((Sx_i^*)(x))^n Ty_i. $$
Consequently,
$$ \Gamma \left(\sum_{i=1}^\infty (x_i^*(\cdot))^n y_i\right) =      \sum_{i=1}^\infty (Sx_i^*(\cdot))^n Ty_i, $$
Note that as we can write $\Gamma(P)$ as $T\circ\overline{P}
  \circ S^*\circ J_E$ it follows that $\Gamma(P)$ is independent of the choice
  of nuclear representation of $P$ and therefore is well-defined.
  Moreover, $\displaystyle \sum_{i=1}^\infty \norma{Sx_i^*}^n \norma{Ty_i} \leq \norma{S}^n \norma{T} \sum_{i=1}^\infty \norma{x_i^*}^n \norma{y_i}$, hence $\Gamma(P)$ is nuclear and
$$ \norma{\Gamma(P)}_{\mathcal{N}} \leq \norma{S}^n \norma{T} \norma{P}_{\mathcal{N}}. $$
Thus $\Gamma$ is a well-defined bounded linear operator. So, we have that 
        \begin{align*}
            \Phi_n (\Psi ( \theta_n(x^*) \otimes y )) = \Phi_n( \theta_n(Sx^*) \otimes Ty ) = (Sx^*(\cdot))^n Ty = \Gamma ((x^*(\cdot))^n y) = \Gamma(\Phi_n(\theta_n(x^*) \otimes y))
        \end{align*}
        holds for all $x^* \in E^*$ and $y \in F$, hence
        $\Phi_n (\Psi ( w_k )) = \Gamma (\Phi_n (w_k)) $ for every $k \in \N$. Therefore
        \begin{align*}
            \Phi_n(\Psi(w)) & = \overline{\Phi}_n(\Psi(w)) = \sum_{k=1}^\infty \overline{\Phi}_n(\Psi(w_k)) = \sum_{k=1}^\infty \Phi_n(\Psi(w_k))\\
            & =  \sum_{k=1}^\infty \Gamma(\Phi_n(w_k)) = \Gamma \left(
            \sum_{k=1}^\infty \Phi_n (w_k)\right) = \Gamma
            \left(\sum_{k=1}^\infty \overline{\Phi}_n(w_k)\right) =
            \Gamma(\overline{\Phi}_n(w)) = 0.
        \end{align*}
        Now, since $\Phi_n$ is a linear isomorphism, we obtain that $\Psi(w) = 0$. From this, we will prove that $w = 0$. First, we notice that, for each
        $k \in \N$
    \begin{align*}
        |(id - \Psi)(w_k)| & \leq |(id - id \otimes T)(w_k)| + |(id \otimes T - \Psi)(w_k)| \\
        & = |id \otimes (id_F - T) (w_k)| + |(id - \otimes^n S) \otimes T (w_k)| \\
        & \leq |id \otimes (id_F - T)| (|w_k|) + |(id - \otimes^n S) \otimes T| (|w_k|) \\
        & \leq  |id \otimes (id_F - T)| \left (    \sum_{j=1}^{m_k} u_{k,j}
     \otimes v_{k,j} \right ) + \\ 
     & +|(id - \otimes^n S) \otimes T| \left (  \sum_{j=1}^{m_k} \sum_{i=1}^{p_{k,j}} \theta_n(g_{k,j}^i)
     \otimes \widetilde{v}_{k,j}^i  \right ),
    \end{align*}
where in the last inequality we used the constructions obtained in Claim 1 and Claim 2 above. We will divide our proof in two steps now. First:
\begin{align*}
    |id \otimes (id_F - T)| \left (    \sum_{j=1}^{m_k} u_{k,j}
     \otimes v_{k,j} \right ) & = \sum_{j=1}^{m_k} |id \otimes (id_F - T)| (u_{k,j}
     \otimes v_{k,j}) \\
     & = \sum_{j=1}^{m_k} u_{k,j} \otimes |(id - T)_{F_{k,j}}|(v_{k,j})
\end{align*}
Second:
\begin{align*}
  |(id - \otimes^n S) \otimes T| & \left (  \sum_{j=1}^{m_k} \sum_{i=1}^{p_{k,j}} \theta_n(g_{k,j}^i)
     \otimes \widetilde{v}_{k,j}^i  \right ) = \sum_{j=1}^{m_k} \sum_{i=1}^{p_{k,j}} |(id - \otimes^n S) \otimes T| (\theta_n(g_{k,j}^i)
     \otimes \widetilde{v}_{k,j}^i) \\
     & \leq  \sum_{j=1}^{m_k} \sum_{i=1}^{p_{k,j}} |id - \otimes^n S|(\theta_n(g_{k,j}^i)) \otimes |T| ( \widetilde{v}_{k,j}^i).
\end{align*}
Applying Lemma \ref{lemaestimate}, we have, for all $k,j, i$,
\begin{align*}
    |(id-\otimes^n S)| (\theta_n(g_{k,j}^i))
&\leq 
\theta_n( g_{k,j}^i+ |(S-id)_{G_{k,j}^i}| (g_{k,j}^i)) -\theta_n(g_{k,j}^i).
\end{align*}
Thus, 
\begin{align*}
    \norma{|id - \otimes^n S|(\theta_n(g_{k,j}^i))} & \leq \norma{\theta_n(g_{k,j}^i + |(S-id)_{G_{k,j}^i}|(g_{k,j}^i)) - \theta_n(g_{k,j}^i)}_{s, |\pi|} \\
    & \leq n (\max \{ \norma{g_{k,j}^i + |(S-id)_{G_{k,j}^i}|(g_{k,j}^i))}, \norma{g_{k,j}^i} \})^{n-1} \, \norma{|(S-id)_{G_{k,j}^i}|(g_{k,j}^i))} \\
    & \leq n (\varepsilon + \frac{1}{k})^{n-1} \varepsilon \leq n 2^{n-1} \varepsilon.
\end{align*}
holds for all $k,j, i$. On the other hand, 
\begin{align*}
    \norma{|T|(\widetilde{v}_{k,j}^i)} & = \norma{|T|(k \norma{\varphi_{k,j}^i})^n y_{k,j}} \leq k^n \norma{\varphi_{k,j}^i}^n \norma{|T|(y_{k,j})} \\
    & = k^n \norma{\varphi_{k,j}^i}^n k \norma{y_{k,j}} \norma{|T|(v_{k,j})} \\
    & \leq k^{n+1} \norma{\varphi_{k,j}^i}^n \norma{y_{k,j}} [\norma{v_{k,j}} + \norma{|(T - id)_{F_{k,j}}|(v_{k,j})}] \\
    & < k^{n+1} \norma{\varphi_{k,j}^i}^n \norma{y_{k,j}} (\frac{1}{k} + \varepsilon) \\
    & \leq \norma{\varphi_{k,j}^i}^n \norma{y_{k,j}} k^n(1 + k)
\end{align*}
holds for all $k,j, i$. Hence,
\begin{align*}
  \sum_{k=1}^\infty \sum_{j=1}^{m_k} \sum_{i=1}^{p_{k,j}} \norma{|T|(\widetilde{v}_{k,j}^i)} & \leq  \sum_{k=1}^\infty\sum_{j=1}^{m_k} \sum_{i=1}^{p_{k,j}} \norma{\varphi_{k,j}^i}^n \norma{y_{k,j}} k^n(1 + k) \\
    & \leq \sum_{k=1}^\infty  \sum_{j=1}^{m_k} \left ( \norma{z_{k,j}}_{s, |\pi|} + \frac{1}{m_k \displaystyle \max_{1 \leq j \leq m_k} \norma{y_{k,j}} \, 2^k} \right ) \norma{y_{k,j}} k^n(1 + k) \\
    & \leq  \sum_{k=1}^\infty \left ( 2^{n-k + 2} k^n(1+k ) +   \sum_{j=1}^{m_k} \frac{\norma{y_{k,j}} k^n (1+k)}{m_k \displaystyle \max_{1 \leq j \leq m_k} \norma{y_{k,j}} 2^k} \right )\\
    & \leq  \sum_{k=1}^\infty \left ( 2^{n-k + 2} k^n(1+k )  + \frac{k^n (1+k)}{2^k} \right )\\
    & =\sum_{k=1}^\infty  k^n(1+k) 2^{-k} (2^{n+2} + 1) := C < \infty.
\end{align*}

Combining the above, we obtain
\begin{align*}
    \norma{w}_{|\pi|} & = \norma{w - \Psi(w)}_{|\pi|} = \norma{(id - \Psi)(\sum_{k=1}^\infty w_k)}_{|\pi|} \leq \sum_{k=1}^\infty \norma{(id - \Psi)(w_k)}_{|\pi|} \\
    & \leq \sum_{k=1}^\infty \sum_{j=1}^{m_k} \norma{u_{k,j}} \norma{|(id - T)_{F_{k,j}}|(v_{k,j})} + \sum_{k=1}^\infty \sum_{j=1}^{m_k}
    \sum_{i=1}^{p_{k,j}} \norma{|id - \otimes^n S|(\theta_n(g_{k,j}^i))} \norma{|T| ( \widetilde{v}_{k,j}^i)}  \\
    & \leq \sum_{k=1}^\infty \sum_{j=1}^{m_k} \norma{u_{k,j}} \varepsilon + \sum_{k=1}^\infty \sum_{j=1}^{m_k} \sum_{i=1}^{p_{k,j}}  n 2^{n-1} \varepsilon \norma{|T| ( \widetilde{v}_{k,j}^i)}  \\
    & = \sum_{k=1}^\infty \sum_{j=1}^{m_k} \norma{u_{k,j}} \varepsilon + n2^{n-1}\varepsilon C.
\end{align*}
   Taking $\varepsilon \longrightarrow 0$, we obtain $\norma{w}_{|\pi|} = 0$, and we have a contradiction.
\end{proof}

\begin{corollary}  \label{cornuclear}
  Let $E, F, G$ be Banach lattices such that $G$ is Dedekind complete and $E^*, F^*$ and $G$ have the LAP. 
  If $E^*$ and $F^*$ are lattice isomorphic, then  $\mathcal{P}_{\mathcal{N}}^r(^n E; G)$ and $\mathcal{P}_{\mathcal{N}}^r(^n F; G)$ are lattice isomorphic.
\end{corollary}

\begin{proof}
   If $E^*$ and $F^*$ are lattice isomorphic, then $\widehat{\otimes}_{n, s, |\pi|} E^\ast$ and $\widehat{\otimes}_{n, s, |\pi|} F^\ast$ 
   are lattice isomorphic. Indeed, if $\varphi\colon E^* \to F^*$ is a
   lattice isomorphism, then $P(x^*) = \theta_n(\varphi(x^*))$ defines a positive $n$-homogeneous polynomial from $E^*$ into $\widehat{\otimes}_{n, s, |\pi|} F^\ast$, and so $P^\otimes\colon \widehat{\otimes}_{n, s, |\pi|} E^\ast \to
   \widehat{\otimes}_{n, s, |\pi|} F^\ast $ is a positive linear operator such that $P^\otimes(\theta_n(x^*)) = \theta_n(\varphi(x^*))$ for every $x^* \in E^*$. Analogously, there exists a positive operator $S\colon
   \widehat{\otimes}_{n, s, |\pi|} F^\ast \to \widehat{\otimes}_{n, s, |\pi|} E^\ast $ such that $S(\theta_n(y^*)) = \theta_n(\varphi^{-1}(y^*))$ for every $y^* \in F^*$. Finally, since
   $$ S \circ P^\otimes (\theta_n(x^*)) = S(\theta_n(\varphi(x^*))) = \theta_n(\varphi^{-1}(\varphi(x^*))) = \theta_n(x^*) $$
   holds for every $x^* \in E^*$, we obtain that $S \circ P^\otimes (z) = z$ for every $z \in \widehat{\otimes}_{n, s, |\pi|} E^\ast$. The same idea also proves that $P^\otimes \circ S (z) = z$
   for every $z \in \widehat{\otimes}_{n, s, |\pi|} F^\ast$. Thus, $(P^\otimes)^{-1} = S$, and by \cite[Theorem 2.15]{alip} $P^\otimes$ is a lattice isomorphism. An adaptation of the above argument can be used to prove that $\widehat{\otimes}_{n, s, |\pi|} E^\ast \,\widehat{\otimes}_{|\pi|} G$ is
  lattice isomorphic to $\widehat{\otimes}_{n, s, |\pi|} F^\ast \,\widehat
  {\otimes}_{|\pi|} G$
. Now, by applying Theorem \ref{teonuclear}, we get that $\mathcal{P}_{\mathcal{N}}^r(^n E; G)$ and $\mathcal{P}_{\mathcal{N}}^r(^n F; G)$ are lattice isomorphic.
\end{proof}

We conclude the article by providing an example where
Corollary \ref{cornuclear} can be applied:
\begin{example}
    It follows from \cite[Remark 4.3]{bupams} that every atomic Banach lattice with an order continuous
norm has the LAP. So, letting $E$ and $F$ be $c_0, c$ or the AM-space $X$ from Example \ref{exteo}, we obtain that $E^*$ and $F^*$ are both lattice isomorphic to $\ell_1$, hence $E^*$ and $F^*$ are lattice isomorphic with the LAP. Now, if $G$ is an arbitrary atomic Banach lattice with order continuous norm, we conclude by Corollary \ref{cornuclear} that $\mathcal{P}^r_{\cal N}(^n E; G)$ and $\mathcal{P}^r_{\cal N}(^n F; G)$ are lattice isomorphic.
\end{example}

\noindent C. Boyd\\
Department of Mathematics and Statistics\\
University College Dublin\\
Belfield -- Dublin 4 -- Ireland\\
e-mail: christopher.boyd@ucd.ie

\medskip

\noindent V. C. C. Miranda\\
Departamento de Matemática\\
Instituto de Ciências Matemáticas e de Computação \\
Universidade de São Paulo \\
13566-590 -- São Carlos - SP -- Brazil  \\
e-mail: viniciusmiranda@icmc.usp.br

\end{document}